\documentclass[english,11pt,a4paper,twoside,reqno]{article}
\usepackage{amsfonts}
\usepackage{amsmath}
\usepackage{amssymb}
\usepackage{amsthm}
\usepackage{hyperref}
\usepackage[svgnames]{xcolor}

\providecommand{\doi}[1]{Eprint \href{http://dx.doi.org/#1}{doi:#1}}

\renewcommand{\geq}{\geqslant}
\renewcommand{\leq}{\leqslant}

\usepackage{eufrak}
\usepackage[T1]{fontenc}
\usepackage{graphics}
\usepackage{color}
\usepackage[english]{babel}
\usepackage[active]{srcltx}
\usepackage{enumerate}
\usepackage{euscript}
\numberwithin{equation}{section}
\usepackage{lineno}

\usepackage{latexsym}
\usepackage{graphicx}
\usepackage{a4wide}
\usepackage{url,float}         

{\theoremstyle{THkey}\newtheorem{theoremofothers}{}}
\newcounter{compteur}
\newcounter{hypothese}
\newcounter{propo}
\theoremstyle{plain}
\newtheorem{proposition}[propo]{Proposition}
\newtheorem{theorem}[compteur]{Theorem}
\newtheorem{lemma}[propo]{Lemma}

\newtheorem{corollary}[compteur]{Corollary}

\theoremstyle{definition}
\newtheorem{definition}[hypothese]{Definition}
\theoremstyle{remark}
\newtheorem{remark}{Remark}
\newtheorem{example}{Example}

\newcommand{\g}{\mathfrak{g}}
\newcommand{\R}{\mathbb{R}}
\newcommand{\rt}{\rightarrow}
\newcommand{\Rn}{\mathbb{R}^p}
\newcommand{\co}{\operatorname{co}}
\newcommand{\dom}{\operatorname{dom}}
\newcommand{\K}{\mathcal{K}}
\newcommand{\N}{\mathbb{N}}

\newcommand{\Esp}{\mathbb{E}_{\sigma,\tau,\omega_1}}
\newcommand{\NEW}[1]{{\em #1}}

\begin{document}


\title{\sc  Definable zero-sum stochastic games}

\author{
J\'{e}r\^{o}me BOLTE\footnote{TSE (GREMAQ, Universit\'e Toulouse Capitole), Manufacture des Tabacs, 21 all\'ee de Brienne, 31015 Toulouse Cedex 5, France. email: {\tt jerome.bolte@tse-eu.fr}} , \; St\'ephane GAUBERT \footnote{INRIA \& Centre de Math\'ematiques Appliqu\'ees (CMAP), UMR 7641, \'Ecole Polytechnique, 91128
Palaiseau, France. email: {\tt Stephane.Gaubert@inria.fr}
} \, \& Guillaume VIGERAL\footnote{Universit\'e Paris-Dauphine, CEREMADE, Place du Mar\'echal De Lattre de Tassigny. 75775 Paris
cedex 16, France. email: {\tt guillaumevigeral@gmail.com} }}
\maketitle
\renewcommand{\thefootnote}{}\footnotetext{
The first and second author were partially supported by the PGMO Programme of Fondation Math\'ematique Jacques Hadamard and EDF. The third author was partially supported by the french Agence Nationale de la Recherche (ANR) "ANR JEUDY: ANR-10-
BLAN 0112."
This work was co-funded by the European
Union under the 7th Framework Programme ``FP7-PEOPLE-2010-ITN'',
grant agreement number 264735-SADCO. }
\renewcommand{\thefootnote}{\arabic{footnote}}

\begin{abstract} Definable zero-sum stochastic games involve a finite number of states and action sets, reward and transition functions that are definable in an o-minimal structure. Prominent examples of such games are finite, semi-algebraic or globally subanalytic  stochastic games. 

We prove that the Shapley operator of any definable stochastic game with separable transition and reward  functions is  definable in the same structure. Definability in the same structure does not hold systematically: we  provide a  counterexample of a stochastic game with semi-algebraic data yielding a non semi-algebraic  but globally subanalytic Shapley operator.

Our definability results on Shapley operators are used to prove that any separable definable game has a uniform value; in the case of polynomially bounded structures we also provide convergence rates. Using an approximation procedure, we actually establish that general zero-sum games with separable definable transition functions have  a uniform value. These results highlight the key role played by the tame structure of transition functions.  As particular cases of our main results, we obtain that stochastic games with polynomial transitions, definable games with finite actions on one side, definable games with perfect information or switching controls have a uniform value. Applications to nonlinear maps arising in risk sensitive control and Perron-Frobenius theory are also given.
\end{abstract}

\medskip


\section{Introduction}

Zero-sum stochastic games have been widely studied since their introduction by Shapley \cite{Shap53} in 1953 (see the  textbooks \cite{sorin02,FiVr97,MSZ,NeSo03} for an overview of the topic). They model long term  interactions between two players with completely opposite interest; they appear in a wealth of domains including computer science, population dynamics or economics. In such games the players face, at each time $n$, a zero-sum game whose data are determined by the state of nature. The evolution of the game is governed by a stochastic process which is partially controlled by both players through their actions, and which determines, at each stage of the game, the state of nature and thus the current game faced by both players. We assume that  the players know the payoffs functions, the underlying stochastic process and the current state; they also observe at each stage the actions played by one each other. They  aim at optimizing their gain over time. This objective depends on specific choices of payoff evaluations and in particular on the choice of a distribution of discount/weighting factors over time.

We shall focus here on two kinds of payoff evaluations which are based on Ces\`aro and Abel means. For any finite horizon time $n$, one defines the ``repeated game" in $n$ stages for which each player aims at optimizing his averaged gain over the frame time $t=1,\ldots,n$. Similarly for any discount rate $\lambda$, one defines the $\lambda$-discounted game for infinite horizon games. Under  minimal assumptions these games have values, and an important issue in Dynamic Games theory is the asymptotic study of these values (see Subsection~\ref{Defsto}).
These aspects have been dealt along two lines:
\begin{itemize}
\item The ``asymptotic approach" consists in the study of the convergence of these values when the players become more and more patient -- that is when $n$ goes to infinity or $\lambda$ goes to 0.
\item The ``uniform value approach", for which one seeks to establish  that, in addition, both players have near optimal strategies that do not depend on the horizon (provided that the game is played long enough).
\end{itemize}

The asymptotic approach is less demanding as there are games \cite{Zamir73} with no uniform value but for which the value does converge to a common limit; the reader is referred to \cite{MSZ} for a thorough discussion on those two approaches and their differences in zero-sum repeated games.

For the asymptotic approach, the first positive results were  obtained in recursive games \cite{Everett57}, games with incomplete information \cite{AuMa95,MeZa71} and absorbing games \cite{Kohlberg74}. In 1976, Bewley and Kohlberg settled, in a fundamental paper \cite{BewlKohl76}, the case of games with finite sets of states and actions. Their proof is based on the observation that the discounted value, thought of as a function of the discount factor, is semi-algebraic,
and that it has therefore a Puiseux series expansion.

Bewley-Kohlberg's result of convergence was later considerably strengthened by Mertens and Neyman who proved \cite{MertNeym81} the existence of a uniform value in this finite framework. Several types of improvements based on techniques of semi-algebraic geometry were developed in \cite{Ney03,Milman02}. Algorithms using an effective version of the Tarski-Seidenberg theorem were recently designed in order to compute either the uniform value \cite{CMH2008} or $\epsilon$-optimal strategies \cite{solanvieille2010}. 

The semi-algebraic techniques used in the proof of Bewley and Kohlberg have long been considered as specifically related to the finiteness of the action sets and it seemed that they  could not be adapted to wider settings. In \cite{Sha08}  the authors consider a special instance of polynomial games but their focus is computational and concerns mainly the estimation of discounted values for a fixed discount rate. In order to go beyond  their result and to tackle more complex games, most researchers have used topological or analytical arguments, see e.g.\ \cite{MNR2009,Renault2012,Rosenberg2000,RoSo01,RoVi2000,Sorin03,Sorin04,SoVitoappear}. The common feature of most of these papers is to study the analytical properties of the so-called Shapley operator of the game in order to infer various convergence results of the values.   This protocol, called the ``operator approach" by Rosenberg and Sorin, grounds on Shapley's theorem which ensures that the dynamic structure of the game is entirely represented by the Shapley operator.
\smallskip

Our paper can be viewed as a ``definable operator approach". In the spirit of Bewley-Kohlberg and Neyman, we identify first a class of potentially ``well-behaved games" through their underlying geometric features (definable stochastic games) and we investigate what this features imply for the Shapley operator (its definability and subsequent properties). By the use of Mertens-Neyman result this implies in turn the existence of a uniform value for a wide range of games (e.g.\ polynomial games). 

Before giving a more precise account of our results, let us describe briefly the topologi\-cal/geo\-me\-tri\-cal framework used in this paper. The rather recent introduction of o-minimal structures as models for a tame topology  (see \cite{Dries98}) is a major source of inspiration to this work. O-minimal structures can be thought of as
an abstraction of semi-algebraic geometry through an axiomatization of its most essential properties. An o-minimal structure consists indeed in a collection of subsets belonging to spaces of the form $\R^n$, where $n$ ranges over $\N$, called {\em definable sets}~(\footnote{Functions are called definable whenever their graph is definable.}). Among other things, this collection is required to be stable by linear projections and its ``one-dimensional" sets must be finite unions of intervals. Definable sets are then shown to share most of the qualitative properties of semi-algebraic sets like finiteness of the number of connected components or differential regularity up to stratification.

Our motivation for studying stochastic games in this framework is double. First, it appears that definability allows one to avoid highly oscillatory phenomena in a wide variety of settings: partial differential equations  \cite{Simon83}, Hamilton-Jacobi-Bellman equations and control theory (see \cite{Trelat06} and references therein), continuous optimization \cite{Ioff09}. We strongly suspect that definability is a simple means to ensure the existence of a value to stochastic games.

Another very  important motivation for working within these structures is their omnipresence in finite-dimensional models and applications~(see e.g.\ \cite{Ioff09} and the last section).

\smallskip

The aim of this article is therefore to consider stochastic games --with a strong focus on their asymptotic properties-- in this o-minimal framework. We always assume that  the set of states is finite and we say that a stochastic game is definable in some o-minimal structure if all its data (action sets, payoff and transition functions) are definable in this structure. The   central issue behind this work is probably:
\begin{itemize}
\item[ ($\mathcal{Q}$)] {\em Do definable stochastic games have definable Shapley operators }?
\end{itemize}
As we shall see this question plays a pivotal role in the study of stochastic games. It seems however difficult to solve it in its full generality and we are only able to give here partial results. We prove in particular that any stochastic game with definable, separable reward and transition functions (e.g.\ polynomial games) yields a Shapley operator which is definable  in the same structure. The separability assumption is important to ensure  definability in the same structure, we indeed describe a  rather simple semi-algebraic game whose Shapley operator  is globally subanalytic but not semi-algebraic. The general question of knowing whether a definable game has a Shapley operator definable in a possibly larger structure remains fully open.

An important consequence of the definability of the Shapley operator is the existence of a uniform value for the corresponding game (Theorem~\ref{conv}). The proof of this result is both based on the techniques and results of \cite{Ney03} and \cite{MertNeym81}. For games having a Shapley operator definable in a polynomially bounded structure, we also show, in the spirit of Milman \cite{Milman02}, that  the rate of convergence is of the form $O(\frac{1}{n^{\gamma}})$ for some positive $\gamma$.

 These results are used in turn to study games with arbitrary continuous reward functions (not necessarily definable), separable and definable transition functions and compact action sets. Using the Stone-Weierstrass and Mertens-Neyman theorems, we indeed establish that such games have a uniform value (Theorem~\ref{t:value}). This considerably generalizes previous results; for instance, our central results imply that:
\begin{itemize}
\item  definable games in which one player has finitely many actions,
\item games with polynomial transition functions,
\item games with perfect information and definable transition functions,
\item games with switching control and definable transition functions,
\end{itemize}
have a uniform value.


The above results  evidence that most of the asymptotic complexity of a stochastic game lies in its dynamics, i.e.\ in its transition function. This intuition has been reinforced by a recent follow-up work by Vigeral~\cite{Vipreprint} which shows, through a counterexample to  the convergence of values, that the o-minimality of the underlying stochastic process is a crucial assumption. The example involves finitely many states, simple compact action sets, and continuous transition and payoffs but the transition functions are typically non definable since they oscillate infinitely many times on a compact set.

We also include an application to a class of maps arising in risk sensitive control~\cite{fleming,hernandez,akian}
and in nonlinear Perron-Frobenius theory (growth minimization in population dynamics). In this context, one considers a self-map
$T$ of the interior of the standard positive cone of $\mathbb{R}^d$, and looks
for conditions of existence of the geometric growth rate
$[T^k(e)]_i^{1/k}$ as $k\to \infty$, where $e$ is an arbitrary vector
in the interior of this cone. This leads to examples of Shapley
operators, namely, the conjugates of $T$ by ``log-glasses'' (i.e., log-log coordinates), that are definable
in the log-exp structure. This is motivated also by tropical geometry~\cite{viro}. The latter can be thought of as a degenerate limit of classical geometry through log-glasses. This limit process is called
``dequantization''; the inverse process sends Shapley operators
to (non-linear) Perron-Frobenius operators.
This shows that the familiar o-minimal structure used in game theory,
consisting of real semi-algebraic sets, is not the only useful one
in the study of Shapley operators. We note that other o-minimal
structures, like the one involving absolutely converging Hahn series
constructed by van den Dries and Speisseger~\cite{vandendries98}, are also
relevant in potential applications to tropical geometry. \label{vdd}



The paper is structured as follows. The first sections give a basic primer on the theory of o-minimal structures and on stochastic games. We introduce in particular definable zero-sum stochastic games and discuss several subclasses of games. The main result of that section is the following: if the Shapley operator of a game is definable in an o-minimal structure, this game has a uniform value.  Since the Shapley operator is itself a one-shot game where the expectation of the future payoffs acts as a parameter, we study one-shot parametric games in Section~\ref{sec-one-shot}. We prove that the value of a parametric definable game is itself definable in two cases: either if the game is separable, or if the payoff is convex.
These results are in turn used in Section~\ref{sectionvaluesto} to prove the existence of a uniform value for several classes of games including separably definable games. We finally point an application to a class
of ``log-exp'' maps arising in population dynamics (growth minimization
problems) and in risk sensitive control.


\section{O-minimal structures}

O-minimal structures  play a fundamental role in this paper; we recall here the basic results that we shall use  throughout the article. Some references on the subject are
 van der Dries \cite{Dries98}, van\,der\,Dries-Miller \cite{Dries-Miller96}, Coste \cite{Coste99}.

For a given $p$ in $\N$, the collection of subsets of $\R^p$ is denoted by ${\mathcal P}(\R^p)$.
\begin{definition}[o-minimal structure, {\cite[Definition\,1.5]{Coste99}}]
\label{d:omin}
An \NEW{o-minimal} structure on $(\R,+,.)$ is a sequence of Boolean
algebras ${\cal O}=(\mathcal{O}_{p})_{p\in \N}$ with $\mathcal{O}_{p}\subset{\mathcal P}(\Rn)$, such that for each $p\in\N$:
\begin{enumerate}\itemsep=1mm
\item[(i)]
if $A$ belongs to $\mathcal{O}_{p}$, then $A\times\R$ and
$\R\times A$ belong to $\mathcal{O}_{p+1}$ ;
\item[(ii)]
if $\Pi:\R^{p+1}\rightarrow\Rn$ is the canonical
projection onto $\Rn$ then for any $A$ in $\mathcal{O} _{p+1}$,
the set $\Pi(A)$ belongs to $\mathcal{O}_{p}$ ;
\item[(iii)]
$\mathcal{O}_{p}$ contains the family of real algebraic
subsets of $\Rn$, that is, every set of the form
\[
\{x\in\Rn:g(x)=0\},
\]
where $g:\Rn\rightarrow\R$ is a real polynomial function ;
\item[(iv)]
the elements of $\mathcal{O}_{1}$ are exactly the finite
unions of intervals.
\end{enumerate}
\end{definition}

\smallskip

\noindent A subset of $\R^p$ which belongs to an o-minimal structure $\cal O$, is said to be definable in $\cal O$ or simply definable. A
mapping $F:S\subset \R^p\rt \R^q$ is called definable (in ${\cal O}$), if its graph $\{(x,y)\in\R^p\times\R^q:y\in F(x)\}$ is definable (in $\cal O$) as a subset of~$\R^p\times\R^q$.
Similarly if $g:\R^p\rightarrow (-\infty,+\infty]$ (resp.\ $g:\R^p\rightarrow [-\infty,+\infty)$) is  a real-extended-valued function, it is called definable (in $\mathcal O$), if its graph
$\{(x,r)\in\R^p\times \R: g(x)=r\}$ is definable (in $\mathcal O$).

\begin{remark}\label{R.tame}\noindent
The smallest o-minimal structure is given by the class $\cal SA$ of {\em
real semi-algebraic }objects(\footnote{This is due to
axiom (iii). Sometimes this axiom is weakened \cite{Dries98}, allowing smaller classes than ${\cal SA
}$, for instance the structure  of \em{semilinear
sets}.}). We recall that  a set $A \subset \R^p$ is called
semi-algebraic if it can be written as
$$A=\,\bigcup_{j=1}^{l}\,\bigcap_{i=1}^{k}\,\{\,x\in \R^p :
g_{ij}(x)=0,\, h_{ij}(x)<0\},$$ where the $g_{ij}, h_{ij} :
\R^p\rightarrow \R$ are real polynomial functions on $\R^p$. The fact
that ${\cal SA }$ is an o-minimal structure stems from the
Tarski-Seidenberg principle (see \cite{Boc98}) which asserts the validity of
item (ii) in this class.
\end{remark}

\smallskip
The following result is an elementary but fundamental consequence of the definition.
\begin{proposition}[{\cite{Dries-Miller96}}]\label{p:image} Let $A\subset \R^p$ and  $g:A\rightarrow\R^q$  be definable objects.\\
(i) Let $B\subset A$ a definable set. Then $g(B)$ is definable.\\
(ii) Let $C\subset\R^q$ be a definable set. Then $g^{-1}(C)$ is definable.
\end{proposition}

One can already guess from the above definition and proposition that definable sets behave qualitatively as semi-algebraic sets. The reader is referred to
\cite{Dries-Miller96,Coste99} for a comprehensive account on the
topic.

\begin{example}[max and min functions]\label{ex}
In order to illustrate these stability properties, let us consider nonempty subsets $A,B$ of $\R^p,\R^q$ respectively,  and  $g:A\times B\rightarrow \R$ a definable function. Note that the projection axiom applied on the graph of $g$ ensures the definability of both $A$ and $B$. Set $h(x)=\inf_{y\in B} g(x,y)$ for all $x$ in $A$ and let us establish the definability of $h$; note that the domain of $h$, i.e.\ $\dom h=\{x\in A:h(x)> -\infty\}$ may be smaller than $A$ and possibly empty. The graph of $h$ is given by
$$\mbox{graph}\, h:=\left\{ (x,r)\in A\times\R: \left(\forall y\in B, g(x,y)\geq r \right)\mbox{ and } \left(\forall \epsilon>0,\exists y\in B, g(x,y)<r+\epsilon \right) \right\}.$$

As explained below, the assertion
\begin{equation}\label{FO1} \left(\left(\forall y\in B, g(x,y)\geq r \right)\mbox{ and } \left(\forall \epsilon>0,\exists y\in B, g(x,y)<r+\epsilon \right)\right),
\end{equation} is called a first order definable formula, but the main point for the moment is to prove that such a formula necessarily describes a definable set.


Consider the sets
\[T=\left\{(x,r)\in A\times \R: \forall \epsilon>0,\exists y\in B, g(x,y)<r+\epsilon \right\},\]
 \[S_0=\left\{(x,y,r,\epsilon)\in A\times B\times \R\times (0,+\infty): g(x,y)-r-\epsilon<0\right\}.\]
$S_0$ is definable by Proposition~\ref{p:image}(ii). We wish to prove that $T$ is definable.
 Projecting $S_0$  via $\Pi(x,y,r,\epsilon)=(x,r,\epsilon)$, one obtains the definable set $S_1=\{(x,r,\epsilon)\in A\times \R\times (0,+\infty): \exists y \in B, g(x,y)-r-\epsilon<0\}.$  Introducing  $\Pi'(x,r,\epsilon)=(x,r)$, we see that    $T$ can be expressed as
 $$\left(A\times \R\right) \setminus \Pi'\left(E\right)$$
\noindent with $E:=\left(A\times \R \times (0,+\infty)\right)\setminus S_1$.
 Since the complement operations preserve definability,  $T$ is definable. Using this type of idea and Definition~\ref{d:omin}, we can prove similarly that
 $$T'=\{(x,r)\in A\times \R: \forall y\in B, g(x,y)\geq r  \}$$
 is definable. Hence $\mbox{graph}\, h=T\cap T'$ is definable and thus $h$ is definable.
\end{example}


The most common method to establish the definability of a set is thus to interpret it as the result of a finite sequence of basic operations on definable sets (projection, complement, intersection, union). This idea is conveniently captured by the notion of a first order definable formula (when no confusion can occurred we shall simply say first order formula). {\em First order definable formulas } are built inductively according to the following rules:
 \begin{itemize}
 \item If $A$ is a definable set, $x\in A$ is a first order definable  formula
 \item If $P(x_1,\ldots,x_p)$ and $Q(x_1,\ldots,x_q)$ are first order  definable formulas then
 ({$\mbox{not }P$}), ($P\mbox{ and }Q$), and ($P\mbox{ or }Q$) are first order definable formulas.
 \item Let $A$ be a definable subset of $\R^p$ and $P(x_1,\ldots,x_p,y_1,\ldots,y_q)$ a first order definable  formula then both
 $$\begin{array}{l}
 (\exists x\in A, P(x,y))\\
 (\forall x\in A, P(x,y))
 \end{array}$$
 are first order definable formulas.
 \end{itemize}
 Note that Proposition~\ref{p:image} ensures that $``g(x_1,\ldots,x_p)=0"$ or $`g(x_1,\ldots,x_p)<0"$ are first order  definable  formulas whenever $g:\R^p\rightarrow\R$ is definable (e.g.\ polynomial). Note also that \eqref{FO1} is, as announced earlier, a first order definable formula.\\
It is then easy to check, by induction, that:
\begin{proposition}[\cite{Coste99}]\label{p:FO1}
If $\Phi(x_1,\ldots,x_p)$ is a first order definable formula, then $\{(x_1,\ldots,x_p)\in\R^p:\Phi(x_1,\ldots,x_p)\}$ is a definable set.
\end{proposition}

\begin{remark}
A rigorous treatment of these aspects of o-minimality can be found in \cite{Mark02}.
\end{remark}

An easy consequence of the above proposition that we shall use repeatedly and in various form is the following.

\begin{proposition}\label{defder}
Let $\Omega$ be a definable open subset of $\R^n$ and $g:\Omega\rightarrow\R^m$ a definable differentiable mapping. Then its derivative $g'$ is definable.
\end{proposition}
\smallskip

There exists many regularity results for definable sets \cite{Dries-Miller96}. In this paper, we  essentially use the following fundamental lemma.

\medskip

\noindent Let ${\cal O}$ be an o-minimal structure on $(\R,+,.)$.

\begin{theoremofothers}[Monotonicity Lemma {\cite[Theorem\,4.1]{Dries-Miller96}}]
Let $f : I\subset \R
\rightarrow \R$ be a definable function and $k\in \N$. Then there
exists a finite partition of $I$ into $l$ disjoint intervals
$I_1,\ldots,I_l$ such that $f$ restricted to each nontrivial
interval $I_j$, $j\in\{1,\ldots,l\}$ is $C^k$ and either strictly
monotone or constant.
\end{theoremofothers}



\medskip

\noindent
We end this section by giving examples of o-minimal structures (see \cite{Dries-Miller96} and references therein).

\smallskip

\noindent
{\bf Examples} (a) {\bf (globally subanalytic sets)}
There exists an o-minimal
structure, that contains all sets of
the form $\{(x,t)\in [-1,1]^p\times\R: f(x)=t\}$ where
$f:[-1,1]^p\rt \R$ ($p\in \N$) is an analytic function that can be
extended analytically on a neighborhood of the box~$[-1,1]^p$.
The sets belonging to this structure are called {\em globally
subanalytic sets}; see \cite{Dries-Miller96} and also \cite{Bie88} for an account on subanalytic geometry.

 For instance the functions
$$\sin : [-a,a]\rightarrow  \R $$ (where $a$ ranges over $\R_+$) are globally subanalytic, while $\sin :\R\rightarrow\R$
 is not (else the set $\sin^{-1}(\{0\})$ would be finite by Proposition~\ref{p:image}(ii) and Definition~\ref{d:omin}(iv)).

\noindent (b) {\bf (log-exp structure)} There exists an
o-minimal structure containing the globally subanalytic sets and the graph of
$\exp\,: \R\rt \R$.

\medskip

We shall also use a more ``quantitative" characteristic of o-minimal structures.

\begin{definition}[Polynomially bounded structures]\label{d:polbon}
An o-minimal structure is called polynomially bounded if for all function $\psi:(a,+\infty)\rightarrow \R$ there exists a positive constant $C$ and an integer $N$ such that $|\psi(t)|\leq Ct^N$ for all $t$ sufficiently large
\end{definition}

The classes of semi-algebraic sets or of globally subanalytic sets are polynomially bounded \cite{Dries-Miller96}, while the log-exp structure is obviously not.

We have the following result in the spirit of the classical Puiseux development of semi-algebraic mappings, which will be used in the proof of Theorem~\ref{conv} below.

\begin{corollary}[\cite{Dries-Miller96}]\label{c:polbon} If $\epsilon>0$ and $\phi:(0,\epsilon)\rightarrow \R $ is definable in a polynomially bounded o-minimal structure there exist $c\in \R$ and $\alpha\in \R$ such that
$$\phi(t)=ct^{\alpha}+o(t^{\alpha}),\quad t\in (0,\epsilon).$$
\end{corollary}


\section{Stochastic games}
\subsection{Definitions and fundamental properties}\label{Defsto}

\noindent
{\bf Stochastic games: definition.} A stochastic game is determined by

\begin{itemize}
\item Three sets: a finite set of {\em states} $\Omega$, with cardinality $d$, and two nonempty  sets of {\em actions} $X\subset \R^{p}$ and $Y\subset \R^{q}$.
\item A {\em payoff function}
$g:\Omega\times X\times Y\rightarrow \R$ and a {\em transition} probability $\rho : \Omega\times X\times Y\rightarrow \Delta(\Omega)$, where $\Delta(\Omega)$ is the set of probabilities over $\Omega$.
\end{itemize}
Such a game is denoted by $(\Omega,X,Y,g,\rho)$. Unless explicitly specified, we will always assume the following, which guarantees that
the finite horizon and discounted values
do
exist.

\begin{center}
\fbox{\begin{minipage}{0.7\textwidth}
Standing assumptions ({$\cal A$}): The reward function $g$ and the transition function $\rho$ are {\em continuous}; both action sets $X,Y$ are
  {\em nonempty compact sets}.\end{minipage}}
\end{center}

\medskip

\noindent
{\bf Strategies and values.} The game is played as follows. At time $n=1$, the state $\omega_1$ is known by both players, player 1 (resp.\ 2) makes a move $x$ in $X$ (resp.\ $y$ in $Y$), the resulting payoff is $g_1:=g(x_1,y_1,\omega_1)$ and the couple $(x_1,y_1)$ is observed by the two players.  The new state $\omega_2$ is drawn according to the probability distribution $\rho(\cdot|x_1,y_1,\omega_1)$, both players observe this new state and can thus play accordingly. This process goes on indefinitely and generates a stream of actions $x_i,y_i$, states $\omega_i$ and payoffs $g_i=g(x_i,y_i,\omega_i)$.  Denote by $H_n=(\Omega\times X\times Y)^n\times \Omega$ the sets of stories of length\footnote{This is the set of histories at the end of the n-th stage, with the convention that $n=0$ before the first stage.} n, $H=\cup_{n\in\N}H_n$ the set of all finite stories and $H_{\infty}=(\Omega\times X\times Y)^{\N}$ the set of infinite stories. A strategy for player 1 (resp.\ player 2) is a mapping $$\sigma:H\rightarrow \Delta(X)\quad(\mbox{resp.\ }\tau:H\rightarrow \Delta(Y)).$$
A triple $(\sigma,\tau,\omega_1)$ defines a probability measure on $H_{\infty}$ whose expectation is denoted $\Esp$.
The stream of payoffs corresponding to the triple $(\sigma,\tau,\omega_1)$ can be evaluated, at time $n$, as
\begin{equation}
\gamma_n(\sigma,\tau,\omega_1)=\frac{1}{n}\left(\Esp\left(\sum_{i=1}^n g_i\right) \right).
\end{equation}
The corresponding game is denoted by $\Gamma_n$;
Assumption ({$\cal A$}) allows us
to apply Sion's Theorem~
\cite[Theorem A.7, p. 156]{sorin02},
which shows that this game has a value $v_n(\omega_1)$ or simply $(v_n)_1$.
When the sequence $v_n=((v_n)_1,\ldots,(v_n)_d)$ converges as $n$ tends to infinity the stochastic game is said to  have an asymptotic value.

Another possibility for evaluating the stream of payoffs is to rely on a discount factor $\lambda\in]0,1]$ and to consider the game $\Gamma_{\lambda}$ with payoff
\begin{equation}
\gamma_\lambda(\sigma,\tau,\omega_1)=\Esp\left(\lambda\sum_{i=1}^{+\infty}(1-\lambda)^{i-1}g_i\right).
\end{equation}
Applying once more Sion result this game has a value which we denote by $v_{\lambda}(\omega_1)$ or simply $(v_{\lambda})_1$. The vector $v_{\lambda}$ is defined as $v_{\lambda}=((v_{\lambda})_1,\ldots,(v_{\lambda})_d)$.
One of the central question of this paper is to find sufficient conditions to have
$$\lim_{n\rightarrow +\infty} v_n=\lim_{\lambda\rightarrow 0,\,\lambda>0}v_{\lambda}.$$


\noindent
{\bf Shapley operator and Shapley's theorem.} Let us now describe the fundamental result of Shapley which provides an interpretation of the value of the games $\Gamma_n$ as rescaled iterates of  a nonexpansive mapping. In the same spirit, the discounted values $v_{\lambda}$ appear as fixed points of a family of contractions.

  Let $(\Omega, X,Y, g,\rho)$ be an arbitrary stochastic game. The Shapley operator associated to such a game is a mapping $\Psi:\R^d\rightarrow \R^d$, whose $k$th component is defined through
\begin{equation}\label{op_shapley}
\Psi_k(f_1,\ldots,f_d)=\max_{\mu \in \Delta(X)}\min_{\nu \in \Delta(Y)} \int_X\int_Y\left[  g(x,y,\omega_k)+ \sum_{i=1}^d \rho(\omega_i|x,y,\omega_k)f_i\right]d\mu(x) \,d\nu(y).
\end{equation}
Observe as before, that the maximum and the minimum can be interchanged in the above formula. The space $\R^d$ can be
thought of as the set of {\em value functions ${\mathcal F}(\{1,\ldots,d\};\R)$}, i.e.\ the functions which map  $\{1,\ldots,d\}\simeq \Omega$ (set of states) to $\R$ (real-space of values). It is known that a self-map $\Psi$ of $\mathbb{R}^d$ can be represented
as the Shapley operator of some stochastic game
--- that does not satisfy necessarily assumption ({$\cal A$}) --
if and only if it preserves the standard
partial order of $\mathbb{R}^d$ and commutes with the addition
of a constant~\cite{kolokoltsov}.
Moreover, the transition probabilities
can be even required to be degenerate (deterministic), see~\cite{singer00,mxpnxp0}.

\smallskip

\begin{theorem}[Shapley, \cite{Shap53}]$\;$\\
(i) For every positive integer $n$, the value $v_n$ of the game $\Gamma_n$ satisfies $v_n=\frac{1}{n}\Psi^n(0)$.\\
(ii) The value $v_{\lambda}$ of the discounted game $\Gamma_{\lambda}$ is characterized by the following fixed point condition
\begin{equation}\label{fixed_point}
v_{\lambda}=\lambda\Psi(\frac{1-\lambda}{\lambda}\,v_{\lambda}).
\end{equation}
\end{theorem}

\smallskip

\noindent
{\bf Uniform value.} A stochastic game is said to have a uniform value $v_{\infty}$ if both players can almost guarantee  $v_{\infty}$ provided that the length of the $n$-stage game is large enough. Formally, $v_{\infty}$ is the uniform value of the game if for any $\epsilon>0$, there is a couple of strategies of each player $(\sigma, \tau)$ and a time $N$ such that, for every $n\geq N$, every starting state $\omega_1$ and every strategies $\sigma'$ and $\tau'$,
\begin{eqnarray*}
\gamma_n(\sigma,\tau',\omega_1)&\geq &v_{\infty}(\omega_1)-\epsilon\\
\gamma_n(\sigma',\tau,\omega_1)&\leq &v_{\infty}(\omega_1)+\epsilon
\end{eqnarray*}

It is straightforward to establish that if a game has a uniform value $v_\infty$, then $v_n$ and $v_\lambda$ converges to $v_\infty$. The converse is not true however, as there are games with no uniform value but for which $v_n$ and $v_\lambda$ converge~\cite{MeZa71}.

 \smallskip
\noindent
{\bf Some subclasses of stochastic games.}\begin{itemize}
\item {\em Markov Decision Processes} : they correspond to one-player stochastic games (the choice of Player 2 has no influence on payoff nor transition). In this case the Shapley operator has the particular form

\begin{equation}
\Psi_k(f_1,\ldots,f_d)=\max_{x \in X} \left[g(x,\omega_k)+ \sum_{i=1}^d \rho(\omega_i|x,\omega_k)f_i\right]
\end{equation}
\noindent for every $k=1,\ldots,d$.

\item  {\em Games with perfect information }:  each state is entirely controlled by one of the player (i.e.\ the action of the other player has no influence on the payoff in this state nor on transitions from this state). In that case, the Shapley operator has a specific form : for any state $\omega_k$ controlled by Player 1,
\begin{equation}
\Psi_k(f_1,\ldots,f_d)=\max_{x \in X} \left[g(x,\omega_k)+ \sum_{i=1}^d \rho(\omega_i|x,\omega_k)f_i\right],
\end{equation}
\noindent and for any state $\omega_k$ controlled by Player 2,
\begin{equation}
\Psi_k(f_1,\ldots,f_d)=\min_{y \in Y} \left[g(y,\omega_k)+ \sum_{i=1}^d \rho(\omega_i|y,\omega_k)f_i\right].
\end{equation}

\item  {\em Games with switching control }:  in each state the transition is entirely controlled by one of the player (i.e.\ the action of the other player has no influence on transitions from this state, but it may alter the payoff). In that case, the Shapley operator has a specific form: for any state $\omega_k$ where the transition is controlled by Player 1,
\begin{equation}
\Psi_k(f_1,\ldots,f_d)=\max_{\mu \in \Delta(X)}\int_X\left[  \min_{y \in Y} g(x,y,\omega_k)+ \sum_{i=1}^d \rho(\omega_i|x,\omega_k)f_i\right] d\mu(x),
\end{equation}
\noindent and for any state $\omega_k$ where the transition is controlled by Player 2,
\begin{equation}
\Psi_k(f_1,\ldots,f_d)=\min_{\nu \in \Delta(Y)}\int_X\left[\max_{x \in X} g(x,y,\omega_k)+\sum_{i=1}^d \rho(\omega_i|y,\omega_k)f_i \right] d\nu(y).
\end{equation}
\end{itemize}

\medskip
\noindent

\begin{remark}\label{remarqueperfectinfo} Recall that we made assumption ({$\cal A$}) in order to prove the existence of $v_\lambda$ and $v_n$. For Markov decision processes and games with perfect information this existence is automatic whenever the payoff is bounded, hence  there is no need to assume continuity of $g$ or~$\rho$.
\end{remark}

\smallskip
\noindent
{\bf Definable stochastic games.}   Let $\cal O$ be an o-minimal structure. A stochastic game is called {\em definable} if both the payoff function and the probability transition are definable functions.\\

Observe in the above definition that the definability of $g$ implies that the action sets are also definable. Note also that the space $\Delta(\Omega)$, is naturally identified to the $d$ simplex and is thus a semi-algebraic set. Hence there is no possible ambiguity when we assume that transition functions are definable.

\medskip

The questions we shall address in the sequel revolve around the following two ones

\begin{enumerate}
\item[(a)] Under which conditions the Shapley operator of a definable game is definable in the {\em same} o-minimal structure?
\item[(b)] If a Shapley operator of a game is definable, what are the consequences in terms of games values?
\end{enumerate}

In the next subsection we answer the  second question in a satisfactory way: if a Shapley operator is definable, then $v_n$ and $v_\lambda$ converge, to the same limit. The first question is more complex and will be partially answered in Section~\ref{sectionvaluesto}

\subsection{Games with definable Shapley operator have a uniform value}

Let $\cal O$ be an o-minimal structure and $d$ be a positive integer.  We recall the following definition: a subset $\K\subset\R^d$ is  called a {\em cone} if it satisfies $\R_+\K\subset \K$.\\
Let  $\|\cdot \|$ be a norm on  $\R^d$. A mapping $\Psi :A\subset\R^d\rightarrow\R^d$ is called {\em nonexpansive} if
$$\|\Psi(f)-\Psi(g) \|\leq \|f-g \|,$$ whenever $f,g$ are in $\R^d$. Let us recall that the Shapley operator of a stochastic game is nonexpansive with respect to the supremum norm (see \cite{sorin02}), norm which is defined as usual by $\|f\|_{\infty}=\max\{ f_i:i=1,\ldots,d\}$.

The following abstract result is  strongly motivated by the operator approach to stochastic games, i.e.\ the approach in terms of Shapley operator (see Sorin \cite{Sorin03}). It grounds on the work of Bewley-Kohlberg \cite{BewlKohl76} and on its refinement by Neyman \cite[Th.~4]{Ney03}, who showed that the convergence of the iterate $\Psi^n(0)/n$ as $n\to \infty$ is guaranteed if the map $\lambda \to v_\lambda$ has bounded variation, and deduced part $(i)$ of the following theorem in the specific case of a semi-algebraic operator \cite[Th.~5]{Ney03}.

\begin{theorem}[Nonexpansive definable  mappings]\label{conv} The vector space  $\R^d$ is endowed with an arbitrary norm $\|\cdot \|$. Let $\K$ be a nonempty definable closed cone of $\R^d$ and $\Psi:\K\rightarrow\K$  a definable
nonexpansive mapping. Then
\begin{itemize}
\item[(i)] There exists $v$ in $\K$, such that for all $f$ in $\K$, the sequence $\frac{1}{n}\Psi^n(f)$ converges to $v$ as $n$ goes to infinity.
\item[(ii)] When in addition $\Psi$ is definable in a polynomially bounded structure there exists $\theta \in ]0,1[$ and $c>0$ such that
$$\| \frac{\Psi^n(f)}{n}-v\| \leq \frac{c}{n^\theta}+\frac{\|f\|}{n},$$
for all $f$ in $\K$.
\end{itemize}

\end{theorem}
\begin{proof}{Proof.}  For any $\lambda\in(0,1]$, we can apply Banach fixed point theorem to define $V_\lambda$ as the unique fixed point
of the map $\Psi((1-\lambda)\,\cdot)$ and set $v_\lambda=\lambda
V_\lambda$ (recall that $\K$ is a cone). The graph of $V_\lambda$ is given by $\{(\lambda,\,f)\in (0,1]\times\K:\Psi((1-\lambda)f)-f=0\}$. Using Proposition~\ref{p:FO1}, we obtain that $\lambda\rightarrow V_{\lambda}$ and
 $\lambda\rightarrow v_{\lambda}$ are definable in $\mathcal O$. Observe also that
\begin{eqnarray*}
 \|V_\lambda\|&=&\|\Psi((1-\lambda)V_\lambda)\|\\
 &\leq&\|\Psi((1-\lambda)V_\lambda)-\Psi(0)\|+\|\Psi(0)\|\\
 &\leq&\|(1-\lambda)V_\lambda\|+\|\Psi(0)\|
 \end{eqnarray*}
 \noindent so that the curve $\lambda\rightarrow v_\lambda$ is bounded by $\|\Psi(0)\|$.  Applying
 the monotonicity lemma to each component of this curve, we obtain that $v_{\lambda}$ is piecewise $C^1$, has a limit as $\lambda$ goes to $0$ which we denote by $v=v_{0}$. In order to establish that
 \begin{equation}\label{eq:int}
 \int_0^1\|\frac{d}{d\lambda}v_\lambda\|\,d\lambda<+\infty,
 \end{equation}
we first observe that there exists a constant $\mu>0$ such that $\|\cdot\|\leq \mu \|\cdot\|_1$.
 It suffices thus to establish that \eqref{eq:int} holds for the specific case of the 1-norm. Applying  simultaneously the monotonicity lemma to the  coordinate functions of $v_{\lambda}$, we obtain the existence of $\epsilon \in (0,1)$ such that $v_{\lambda}$ is in $C^1(0,\epsilon)$ and such that each coordinate is monotonous on this interval.

   This shows that
\[
\int_0^{\epsilon} \left\|\frac{d}{d\lambda}v_\lambda\right\|_{1} d\lambda=\sum_{i=1}^d
\int_0^{\epsilon}|v'_{\lambda}(\omega_i)| d\lambda=\sum_{i=1}^d \left|(v_{\epsilon})(\omega_i)-(v_{0})(\omega_i)\right|=\|v_{\epsilon}-v_0\|_1,
\]
 and \eqref{eq:int} follows.

   Let  $\bar \lambda$ such that $\lambda\rightarrow v_{\lambda}$ is $C^1$ on $(0,\bar\lambda)$. Let $\lambda>\mu$ be in $(0,\bar\lambda)$.
   Then  for any decreasing sequence  $(\lambda_i)_{i\in \N}$ in $(\lambda,\mu)$, we have
   \begin{eqnarray}\label{estim1}
\sum_{i=1}^{+\infty} \|v_{\lambda_{i+1}}-v_{\lambda_i}\|\leq \int_{\mu}^{\lambda}\|\frac{d}{d\lambda} v_{\lambda}\|ds.
\end{eqnarray}
Indeed
$\|v_{\lambda_{i+1}}-v_{\lambda_i}\|\leq\| \int_{\lambda_{i+1}}^{\lambda_i}\frac{d}{d\lambda}v_{\lambda}d\lambda\|\leq \int_{\lambda_{i+1}}^{\lambda_i}\|\frac{d}{d\lambda}v_{\lambda}\|d\lambda,$
so that the result follows by summation.

The map $\lambda\rightarrow v_\lambda$ is thus of bounded variation, and $(i)$ follows from Neyman's proof that the latter property implies the convergence of $\Psi^n(0)/n$ to the limit $v:=\lim_{\lambda\to 0^+} v_\lambda$~\cite{Ney03}. Some intermediary results in Neyman's proof
are necessary to establish the rate of convergence of $(ii)$;
we thus include the
remaining part of the proof of $(i)$. First observe that
\begin{equation}\label{estim0}\|\frac{1}{n}\Psi^n(f)-\frac{1}{n}\Psi^n(0)\|\leq \frac{1}{n}\|f\|,\; \forall f\in \K \end{equation} for all positive integers $n$, so it suffices to establish the convergence result for $f=0$.

For $n$ in $\N$, define
\[ d_n:=\|nv_{1/n}-\Psi^n(0)\|=\|V_{1/n}-\Psi^n(0)\|,
\]
and let us prove that $n^{-1}d_n$ tends to zero as $n$ goes to infinity.
If $n>0$, we have
\begin{align}
d_{n}&=\|\Psi((n-1)v_{1/n})-\Psi^n(0)\|\nonumber \\
&\leq\|(n-1)v_{1/n}-\Psi^{n-1}(0)\|\nonumber\\
&\leq d_{n-1} + (n-1)\|v_{1/n}-v_{1/n-1}\|.\label{e-iterate}
\end{align}
Let
\[
D_n := \sum_{i\geq n}\|v_{1/i+1}-v_{1/i}\|  <\infty \enspace,
\]
Using~\eqref{e-iterate} and a discrete integration by parts,
we get
\begin{align}
d_n&\leq
 \sum_{i=1}^{n-1}
 i \|v_{1/i+1}-v_{1/i}\|+d_1 \\
&=  \sum_{i=1}^{n-1}
 i (D_i - D_{i+1})+d_1
=
\sum_{i=1}^{n-1}D_i -(n-1)D_n + d_1  \enspace .
\label{e-abel}
\end{align}
Since $D_n$ tends to $0$ as $n\to\infty$, the Ces\`aro
sum $n^{-1}\sum_{i=1}^{n-1}D_i$ also tends to $0$ as $n\to\infty$.
Then, it follows from~\eqref{e-abel} that $n^{-1}d_n$ tends
to $0$ as $n\to\infty$.


Finally $\left\|v_{1/n}-\frac{1}{n}\Psi^n(0)\right\|$ tends to 0
as $n$ goes to infinity. We know from the monotonicity lemma
(or from the fact that $v_\lambda$ as bounded variation) that $v_{1/n}$ converges
to some $v$. It follows that $\frac{1}{n}\Psi^n(0)$ also converges to $v$.

\medskip

We now prove (ii). Assume that $\Psi$ is definable in a polynomially bounded structure and recall that the monotonicity lemma implies the existence of $\bar \lambda\in (0,1)$ such that
$v$ is $C^1$ on the open interval $(0,\bar \lambda)$ and continuous on $[0,\bar\lambda)$.  Since $\mathcal O$ is a polynomially bounded structure and since the first derivative of $v$ is also definable in $\cal O$ (see Proposition~\ref{defder}), there exist $\gamma$ and $c_1\geq 0$ such that $\|\frac{d}{d\lambda}v_{\lambda}\|= c_1\lambda^{-\gamma} + o(\lambda^{-\gamma})$ (see Corollary~\ref{c:polbon}). If we are able to deal with the case when $\gamma$ is positive, the other case follow trivially. Assume thus that $\gamma$ is positive; note that, since $\frac{d}{d\lambda}v_{\lambda}$ is integrable, we must also have $\gamma<1$. Let $c_2>0$ be  such that
$$ \|\frac{d}{d\lambda} v_{\lambda}\|\leq c_2\lambda^{-\gamma},$$
for all positive $\lambda$ small enough.
Let us now consider a positive integer $i$ which is sufficiently large; by using \eqref{estim1}, we have
\begin{eqnarray}
i\|v_{1/i}-v_{1/i+1}\| & \leq & i \int_{\frac{1}{i+1}}^{\frac{1}{i}}\|\frac{d}{d\lambda}v_{\lambda}\|d\lambda\\ \nonumber
             & \leq & i \int_{\frac{1}{i+1}}^{\frac{1}{i}}c_2\lambda^{-\gamma}d\lambda\\ \nonumber
             & \leq & \int_{\frac{1}{i+1}}^{\frac{1}{i}}c_2\lambda^{-1}\lambda^{-\gamma}d\lambda\\ \nonumber
&= & c_2 \left[\frac{1}{-1-\gamma}\lambda^{-\gamma}\right]_{\frac{1}{i+1}}^{\frac{1}{i}}\\
                            & = & \frac{c_2}{1+\gamma} \left((i+1)^{\gamma}-i^{\gamma}\right) \label{maj}
\end{eqnarray}
Replacing $c_2$ by a bigger constant, we may actually assume that \eqref{maj} holds for all positive integers. Hence
\begin{eqnarray*}
||v_{\frac{1}{n}}-\frac{\Psi^n(0)}{n}||=n^{-1}d_n &\leq &n^{-1}\sum_{i=1}^{n} i \|v_{1/i+1}-v_{1/i}\|-n^{-1}d_1\\
                & \leq &n^{-1}\sum_{i=1}^{n} \frac{c_2}{1+\gamma} (i^{\gamma}-(i+1)^{\gamma})-n^{-1}d_1\\
                   & \leq & \frac{c_2}{1+\gamma}\frac{(n+1)^{\gamma}}{n}-n^{-1}d_1\\
                   &= & O\left(\frac{1}{n^{1-\gamma}}\right).
\end{eqnarray*}
Recalling the estimate~\eqref{estim0} and  observing that

\begin{eqnarray*}
\| \frac{\Psi^n(0)}{n}-v\| & \leq & \| \frac{\Psi^n(0)}{n}-v_{\frac{1}{n}}\| +\| v_{\frac{1}{n}}-v\|\\
    & = & O\left(\frac{1}{n^{1-\gamma}}\right)+\int_0^{\frac{1}{n}}\|\frac{d}{d\lambda}v_{\lambda}\|d\lambda\\
    & \leq & O\left(\frac{1}{n^{1-\gamma}}\right)+\int_0^{\frac{1}{n}}c_2\frac{1}{\lambda^{\gamma}}d\lambda=O\left(\frac{1}{n^{1-\gamma}}\right)
\end{eqnarray*}
the conclusion follows by setting $\theta=1-\gamma$ ($\theta\in(0,1)$). 
\end{proof}

\smallskip

The above result and some of its consequences can be recast within game theory as follows. Point (iii) of the following corollary is essentially due to Mertens-Neymann  \cite{MertNeym81}.

\begin{corollary}[Games values and Shapley operators]\label{c:shapdef} If the Shapley operator of a stochastic game is definable the following assertions hold true.
\begin{enumerate}
\item[(i)]  The limits of $v_{\lambda}$ and $v_n$ coincide, i.e.\
$$\lim_{n\rightarrow+\infty}v_n=\lim_{\lambda\rightarrow 0} v_{\lambda}:=v_{\infty}.$$
\item[(ii)] If $\Phi$ is definable in a polynomially bounded o-minimal structure, there exists $\theta \in (0,1]$ such that
$$\|v_n-v_{\infty}\|=O(\frac{1}{n^{\theta}}).$$
\item[(iii)] {\bf (Mertens-Neyman, \cite{MertNeym81})} The game has a uniform value.
\end{enumerate}
\end{corollary}
\noindent
\begin{proof}{Proof.}
Since the Shapley Operator of a game is nonexpansive for the supremum norm, the two first points are a mere rephrasing of the proof of Theorem~\ref{conv}. Concerning the last one, we note from the proof (see \eqref{eq:int}), that there exists an $L^1$ definable function $\phi:(0,1)\rightarrow \R_+$ such that
\begin{equation}\label{eq:boundvar}
\|v_{\lambda}-v_{\mu}\|\leq\int_{\lambda}^{\mu}\phi(s)ds,
\end{equation}
whenever $\lambda<\mu$ are in $(0,1)$. Applying \cite[Theorem of p. 54]{MertNeym81}, the result follows (\footnote{In \cite{MertNeym81} the authors uniquely consider {\em finite} stochastic games, however their proof relies only on the property~\eqref{eq:boundvar}. We are indebted to X.~Venel for his valuable advices on this aspect.}). 
\end{proof}

\begin{remark}
The first two items of Corollary~\ref{c:shapdef} remain true if we do not assume that players observe the actions (since the value $v_\lambda$ does not depend on this observation). Similarly the third item remains true if players only observe the sequence of states and the stage payoffs.
\end{remark}

\begin{remark}[Stationary strategies]\label{r:stat}
When the action sets are infinite, we do not know in general if  the correspondences of optimal stationary actions in the discounted game,
$$\lambda\rightarrow X_\lambda(\omega_i),\:\: \lambda\rightarrow Y_\lambda(\omega_i),\,i=1,\ldots,d,$$
are definable. However, in the particular case of games with perfect observation, the existence of optimal \emph{pure} stationary strategies ensures that for each state $\omega_i$ the above correspondence are indeed definable.
\end{remark}

\begin{remark}[Regularity of definable Shapley operators]
In the particular case of finite games,   more is known: it is proved in \cite{Milman02} that the real $\theta$ in (ii) can be chosen depending only on the dimension (number of states and actions) of the game. These global aspects cannot be deduced directly from our abstract approach in Theorem~\ref{conv}. However we think that similar results could be derived for {\em definable families of Shapley operators induced by definable families of games } as those described in Section~\ref{sectionvaluesto}.
\end{remark}

\begin{remark}[Semi-smoothness of Shapley operators]\label{qi}
The definability of the Shapley operator
 and its Lipschitz continuity imply by \cite[Theorem 1]{BDL09} its semi-smoothness. Since the works of Qi and Sun \cite{qiliqun}, the semi-smoothness condition has been identified as an essential ingredient behind the  good local behavior of nonsmooth Newton's methods. We think that this type of regularity might help game theorists in designing/understanding algorithms for computing values of stochastic games. Interested readers are referred to  \cite[Section 3.3]{FiVr97} for related topics and possible links with iterating policy methods.
\end{remark}
\section{Definability of the value function for parametric games}\label{sec-one-shot}


Let $\mathcal O$ be an o-minimal structure over $(\R,+,.)$. The previous section showed the importance of proving the definability of the Shapley operator of a game.\\

 Recall that the Shapley operator associates to each vector $f$ in $\R^d$, the values of $d$ zero-sum games
$$\max_{\mu\in \Delta(X)}\min_{\nu \in \Delta(Y)} \int_X\int_Y\left[ g(x,y,\omega_k)+ \sum_{i=1}^d \rho(\omega_i|x,y,\omega_k)f_i\right]d\mu d\nu,$$
where $k$ ranges over $\{1,\ldots,d\}$.
 Hence each coordinate function of the operator can be seen as the value of a static zero-sum game depending on a vector parameter $f$. In this section we thus turn our attention to the analysis of {\em parametric zero-sum games} with definable data.

\smallskip

Consider nonempty compact sets $X\subset \R^p,Y\subset\R^q$, an arbitrary nonempty set $Z\subset\R^d$ and a continuous pay-off function
$\g:X\times Y \times Z  \rightarrow  \R.$ The sets
$X$ and $Y$ are action spaces for players 1 and 2, whereas $Z$ is a {\em parameter space}. Denote by $\Delta(X)$ (resp.\ $\Delta(Y)$) the set of probability measures over $X$ (resp.\ $Y$).  When $z\in Z$ is fixed, the mixed extension of $g$ over  $\Delta(X)\times\Delta(Y)$ defines a zero-sum game $\Gamma(z)$ whose value is denoted by $V(z)$ (recall that the $\max$ and $\min$ commutes by Sion's theorem):

\begin{eqnarray}
V(z)&=&\max_{\mu\in \Delta(X)}\min_{\nu \in \Delta(Y)} \int_X\int_Y \g(x,y,z) d\mu d\nu\\
&=& \min_{\nu \in \Delta(Y)} \max_{\mu\in \Delta(X)} \int_X\int_Y \g(x,y,z) d\mu d\nu.
\end{eqnarray}

 In the sequel a parametric zero-sum  game is denoted by $(X,Y,Z,\g)$; when the objects $X,Y,Z,\g$ are definable, the game $(X,Y,Z,\g)$ is called {\em definable}.

The issue we would like to address in this section is: can we assert that the value function $V:Z\rightarrow \R$ is definable in $\mathcal O$ whenever the game $(X,Y,Z,\g)$ is definable in $\mathcal O$?

\smallskip

 As shown in  a forthcoming section, the answer to the previous question is not positive in general; but as we shall see additional algebraic or geometric structure may ensure the definability of the value function.

\subsection{Separable parametric games}

The following type of games and the ideas of convexification used in their studies seems to originate in the work of Dresher-Karlin-Shapley \cite{DresKarlShap50} (where these games appear as {\em polynomial-like games}).

When $x_1,\ldots,x_m$ are vectors in $\R^p$, the convex envelope of the family $\{x_1,\ldots,x_m\}$ is denoted by
$$\mbox{co}\,\{x_1,\ldots,x_m\}.$$

\begin{definition}[Separable functions and games]{\rm Let
$X\subset\R^p$, $Y\subset \R^q$, $Z\subset \R^d$ and\\
$\g:X\times Y\times Z\rightarrow \R$ be as above.\\
(i)  The function $\g$  is called {\em separable} with respect to  the variables $x,y$, if it is of the form
\[
\g(x,y,z)=\sum_{i=1}^{I}\sum_{j=1}^{J}  m_{ij}(z)a_i(x,z) b_j(y,z).
\]
where $I,J$ are positive integers and the $a_i$, $b_j$, $m_{ij}$ are continuous functions. \\
The function $\g$ is called {\em separably definable,} if in addition the functions  $a_i$, $b_j$, $m_{ij}$ are definable.\\
(ii) A parametric game  $(X,Y,Z,\g)$ is called {\em separably definable}, if its payoff function $\g$ is itself separably definable.}
\end{definition}

\begin{proposition}[Separable definable parametric games]\label{p:sep} Let $(X,Y,Z,\g)$ be a separably definable zero-sum game. Then the value function $Z\ni z\rightarrow V(z)$ is definable in $\mathcal O$.
\end{proposition}
\begin{proof}{Proof.}
Let us consider the correspondence
 ${\mathcal L}:Z \rightrightarrows \R^{I}$ defined by
 $${\mathcal L}(z)=\co \{(a_1(x,z),\cdots,a_I(x,z)) : x\in X\}$$ and define
$\mathcal M:Z \rightrightarrows \R^J$
similarly by
${\mathcal M}(z)=\co \{(b_1(y,z),\cdots,b_J(y,z)) : y\in Y\}.$ Using Carath\'eodory's theorem, we observe that the graph of $\cal L$ is defined by a first order formula, as $(z,s)\in \mbox{graph}\,{\cal L}\subset Z\times \R^I$ if and only if
\[ \exists (\lambda_1,\ldots,\lambda_{I+1})\in \R_+^{I+1},\exists (x_1,\ldots,x_{I+1})\in X^{I+1}, \sum_{i=1}^{I+1}\lambda_i=1, s=\sum_{i=1}^{I+1}\lambda_i a_i(x_i,z)
\enspace .\]
 This ensures the definability of $\cal L$ and $\cal M$. Let us introduce the definable matrix-valued function
 $$Z\ni z\rightarrow M(z)=[m_{ij}(z)]_{1\leq i \leq I\,, 1\leq i \leq J}$$ and the mapping
$$W(z)=\sup_{S\in{\mathcal L}(z)}\inf_{T\in{\mathcal M}(z)}\, SM(z)T^t.$$
Using again Proposition~\ref{p:FO1}, we obtain easily that $W$ is definable. Let us prove that $W=V$, which will conclude the proof.
Using the linearity of the integral
\begin{eqnarray*}
W(z)=\sup_{S\in{\mathcal L}(z)}\inf_{T\in{\mathcal M}(z)}SM(z)T^t
&=&\sup_{S\in{\mathcal L}(z)}\inf_{y\in  Y} \sum_{i=1}^I\sum_{j=1}^J m_{ij}(z)S_i \, b_j(y,z)\\
&\leq & \sup_{\mu\in\Delta(X)}\inf_{y\in Y} \int_X \g(x,y,z) d\mu \\
&=&V(z).
\end{eqnarray*}
An analogous inequality for $\inf\sup$ and a minmax argument imply
the result. 
\end{proof}


\subsection{Definable parametric games with convex payoff}

Scalar products on $\R^m$  spaces are denoted by $\langle\cdot,\cdot\rangle$. \\
We consider parametric games $(X,Y,Z,\g)$ such that:
 \begin{equation}\label{e:convex} Y \mbox{ and the partial payoff }\g_{x,z}:\left\{\begin{array}{lll} Y& \rightarrow & \R\\
  y & \rightarrow & \g(x,y,z)
  \end{array}\mbox{ are both convex.}\right.\end{equation}
One could alternatively assume that $X$ is convex and that player 1 is facing a concave function $\g_{y,z}$ for each $y,z$ fixed.

We recall some well-known concepts of convex analysis (see \cite{Rock70}). If $f:\R^p\rightarrow (-\infty,+\infty]$ is a convex function
 its {\em subdifferential} $\partial f(x)$ at $x$ is defined by
 $$x^*\in\partial f(x)\Leftrightarrow f(y)\geq f(x)+\langle x^*,y-x\rangle, \forall y\in \R^p,$$
 whenever $f(x)$ is finite; else we set $\partial f(x)=\emptyset$.
 When $C$ is a closed convex set and $x\in C$, the {\em normal cone} to $C$ at $x$ is given by
 $$N_C(x):=\left\{v\in\R^p: \langle v,y-x\rangle\leq 0, \forall y\in C\right\}.$$
 The indicator function of $C$, written $I_C$, is defined by $I_C(x)=0$ if $x$ is in $C$, $I_C(x)=+\infty$ otherwise. It is straightforward
 to see that $\partial I_C=N_C$ (where we adopt the convention $N_C(x)=\emptyset$ whenever $x\notin C$).

 \smallskip

\begin{proposition}\label{convexparametric} Let $(X,Y,Z,\g)$ be a zero-sum parametric game. Recall that $X\subset \R^p, Y\subset \R^q$ are nonempty compact sets and $\emptyset\neq Z\subset\R^d$ is arbitrary.

Assume that $Y$ and $\g$ satisfy \eqref{e:convex}. Then
\begin{enumerate}
\item[(i)] The value $V(z)$ of the game coincides with
$$\max_{\begin{array}{r}(x_1,\ldots,x_{q+1}) \in X^{q+1}\\ \lambda\in\Delta_{q+1}\end{array}}\;\min_{\begin{array}{r}y\in Y\\  \end{array}} \: \, \sum_{i=1}^{q+1} \lambda_i \g(x_i,y,z),$$
where $\Delta_{q+1}=\{(\lambda_1,\ldots,\lambda_{q+1}) \in \R_+:\sum_{i=1}^{q+1}\lambda_i=1\}$ denotes the $q+1$ simplex.
\item[(ii)] If the payoff function $\g$ is definable then so is the value mapping $V$.
\end{enumerate}
\end{proposition}
\begin{proof}{Proof.}
Item (ii) follows from the fact that (i) provides a first order formula that describes the graph of $V$.

Let us establish (i).  In what follows $\partial $ systematically denotes the subdifferentiation with respect to the variable $y\in Y$, the other variables being fixed.

Fix $z$ in the parameter space. Let us introduce the following continuous function
\begin{equation}
\Phi(y,z)=\max_{x\in X} \g(x,y,z).
\end{equation}
$\Phi(\cdot,z)$ is clearly convex and  continuous. Let us denote by $\bar y$ a minimizer of $\Phi(\cdot,z)$ over $Y$.
Using the sum rule for the subdifferential of convex functions, we obtain
\begin{equation}
\partial \Phi(\bar y,z)+N_Y(\bar y)\ni 0.
\end{equation}
Now from the envelope's theorem (see \cite{Rock70}), we know that
$\partial \Phi(\bar y,z)=\co \{\partial \g(x,\bar y,z):x\in J(y,z)\},$ where $J(y,z):=\{x \mbox{ in }X \mbox{ which maximizes } \g(x,y,z)\mbox{ over }X\}$. Hence Carath\'eodory's theorem implies the
existence of $\mu$ in the simplex of $\R^{q+1}$, $x_1,\ldots,x_{q+1}\in X$ such that
 \begin{equation}\label{opt}
\sum_{i=1}^{q+1}\mu_i\partial \g(x_i,\bar y,z)+N_Y(\bar y)\ni 0.
\end{equation}
where, for each $i$, $x_i$ is a maximizer of $x\rightarrow \g(x,\bar y,z)$ over the compact set $X$. Being given $x$ in $X$, the Dirac measure at $x$ is denoted by $\delta_x$.
We now establish that $\bar x=\sum_{i=1}^{q+1}\mu_i\delta_{x_i}$ and $\bar y$ are optimal strategies in the game
$\Gamma(z)$. Let $x$ be in $X$, we have
\begin{eqnarray}\label{eq:saddle}
\int_X \g(s,\bar y,z) d\bar x(s)& = &\sum_i\mu_i\g(x_i,\bar y,z)\\  \nonumber
                               & = &  \sum_i\mu_i\g(x_1,\bar y,z)\\   \nonumber
                              & = &  \g(x_1,\bar y,z)\\  \nonumber
                              &\geq & \g(x,\bar y,z).
\end{eqnarray}
Using the sum rule for the subdifferential, we see that \eqref{opt} rewrites
$$\partial \left(\sum_i\mu_i \g(x_i,\cdot,z)+I_Y\right)(\bar y)\ni 0,$$
where $I_Y$ denotes the indicator function of $Y$. The above equation implies that $\bar y$ is a minimizer of the convex function
$\sum_i\mu_i \g(x_i,\cdot,z)$ over $Y$. This implies that
\begin{eqnarray*}
\int_X \g(s,\bar y,z) d\bar x(s)& = &\sum_i\mu_i \g(x_i,\bar y,z)\\
                               & \leq &  \sum_i\mu_i \g(x_i, y,z)
\end{eqnarray*}
for all $y$ in $Y$.
Together with \eqref{eq:saddle}, this shows that $(\bar x,\bar y)$ is a saddle point of the mixed extension of $\g$ with value $\int_X \g(s,\bar y,z)d\bar x(s)$. To conclude, we finally observe that we also have  $$\sum_{i=1}^{q+1}\mu_i \g(\bar x_i,\bar y,z)=\g(\bar x_1,\bar y,z) \geq \sum_{i=1}^{q+1}\lambda_i \g(x_i,\bar y,z)$$ for all $\lambda\in \Delta_{q+1}$ and $x_i$ in $X$. Hence $\left((\lambda,x_1,\ldots,x_{q+1}),\bar y\right)$
is a saddle point of the map $\left((\lambda,x_1,\ldots,x_{q+1}),y\right)\rightarrow \sum_{i=1}^{q+1}\lambda_i \g(x_i,y,z)$
with value $\sum\mu_i\g(\bar x_i,\bar y,z)=\int_X \g(s,\bar y,z)d\bar x(s)$. 
\end{proof}

\begin{remark}{\rm
(a) Observe that the above proof actually yields optimal strategies for both players.\\
(b) An analogous result holds, when we assume that $X$ is convex and $X\ni x\rightarrow g(x,y,z)$ is a concave function.}
\end{remark}

\subsection{A semi-algebraic parametric game whose value function is not semi-algebraic}
The following lemma is adapted from an example in McKinsey \cite[Ex.~10.12 p 204]{McKinsey} of a one-shot game played on the square where the payoff is a rational function yet the value is transcendental.

\begin{lemma}\label{lemmecontreexemple}
Consider the semi-algebraic  payoff function $$\g(x,y,z)=\frac{(1+x)(1+yz)}{2(1+xy)^2}$$ where $(x,y,z)$ evolves in $[0,1]\times[0,1]\times(0,1]$. Then $$V(z)=\frac{z}{2\ln(1+z)}, \quad \forall z\in(0,1].$$
\end{lemma}
\begin{proof}{Proof.}
Fix $z$ in $(0,1]$.
Player 1 can guarantee $V(z)$ by playing the probability density
\[ \frac{dx}{\ln(1+z)(1+x)}
\]
on $[0,z]$ since for any $y\in[0,1]$,
\[
\int_0^z \frac{\g(x,y,z) dx}{\ln(1+z)(1+x)} = \frac{1+yz}{2\ln(1+z)} \int_0^z \frac{dx}{(1+xy)^2} =\frac{z}{2\ln(1+z)}
\]

On the other hand, Player 2 can guarantee $V(z)$ by playing the probability density
\[\frac{z\, dy}{\ln(1+z)(1+yz)}
\]
on $[0,1]$ since for any $x\in[0,1]$,
\[
\int_0^1 \frac{z\, \g(x,y,z) dy}{\ln(1+z)(1+yz)} = \frac{z(1+x)}{2\ln(1+z)} \int_0^1 \frac{dy}{(1+xy)^2} =\frac{z}{2\ln(1+z)}.
\qquad
\]
\end{proof}

We see on this example that the underlying objects of the initial game are semi-algebraic while the value function is not. Observe however that the value function is definable in a {\em larger structure} since it is globally subanalytic (the $\log$ function only appears through its restriction on compact sets). The question of the possible definability of the value function in a larger structure is exciting but it seems difficult, it is certainly a matter for future research.

\section{Values of stochastic games}\label{sectionvaluesto}
\subsection{Definable stochastic games}\label{propperfect}
We start by a simple result. Recall that a stochastic game has perfect information if each state is controlled by only one of the players (see Section~\ref{Defsto}).

\begin{proposition}[Definable games with perfect information]\label{p:perfect} Definable games with perfect information and bounded payoff {\rm (\footnote{Recall that we do not need to assume continuity of $g$ and $\rho$ in that case, as stated in Remark \ref{remarqueperfectinfo}})} have a uniform value.
\end{proposition}
\begin{proof}{Proof.}
 Let $\omega_k$ be any state controlled by the first player. The Shapley operator in this state can be written as
\[
\Psi_k(f)=\sup_{X} \left[  g(x,\omega_k)+ \sum_{i=1}^d \rho(\omega_i|x,\omega_k)f_i\right].
\]
So $\Psi_k$ is the supremum, taken on a definable set, of definable functions, and is thus definable (see Example~\ref{ex}). The same is true if $\omega_k$ is controlled by the second player, so we conclude by Corollary~\ref{c:shapdef}.~
\end{proof}

A stochastic game $(\Omega,X,Y,g,\rho)$ is called {\em separably definable}, if both the payoff and the transition functions are separably definable. More precisely:
\begin{enumerate}
\item[(a)] $\Omega$ is finite and $X\subset\R^p$, $Y\subset\R^q$ are definable sets.
\item[(b)] For each state $\omega$, the reward function $g(\cdot,\cdot,\omega)$ has a definable/separable structure, that is

$$g(x,y,\omega):=\sum_{i=1}^{I_{\omega}}\sum_{j=1}^{J_{\omega}}m_{i,j}^{\omega}\,a_i(x,\omega)\,b_j(y,\omega),\;\forall (x,y)\in X\times Y,$$
where  ${I_{\omega}},{J_{\omega}}$ are positive integers, $m_{ij}^{\omega}$ are real numbers, $a_i(\cdot,\omega)$ and $b_j(\cdot,\omega)$ are continuous definable functions.
\item[(c)] For each couple of states $\omega,\omega'$, the transition function $\rho(\omega'|\cdot,\cdot,\omega)$ has a  definable/separable structure, that is
$$\rho(\omega'|x,y,\omega):=\sum_{i=1}^{K_{(\omega,\omega')}}\sum_{j=1}^{L_{(\omega,\omega')}} n_{i,j}^{(\omega,\omega')}\;c_i(x,\omega,\omega')\,d_j(y,\omega,\omega')\;\forall (x,y)\in X\times Y,$$
where $K_{(\omega,\omega')},L_{(\omega,\omega')}$ are positive integers,  $n_{ij}^{(\omega,\omega')}$ are  real numbers, $c_i(\cdot,\omega,\omega')$ and $d_j(\cdot,\omega,\omega')$ are continuous definable functions.
\end{enumerate}
The most natural example of separably definable games are games with semi-algebraic action spaces and polynomial reward and transition functions.

\begin{theorem}[Separably definable games]\label{t:sepgamesunif}
Separably definable games have a uniform value.
\end{theorem}
\begin{proof}{ Proof.}
The coordinate functions of the Shapley operator yield $d$ parametric separable definable games. Hence the Shapley operator of the game, say $\Psi$, is itself  definable by Proposition~\ref{p:sep}.  Applying Corollary~\ref{c:shapdef} to $\Psi$, the result follows. 
\end{proof}

\smallskip

An important subclass of separable definable games is the class of definable games for which {\em one of the player} has a finite set of strategies.
\begin{corollary}[Definable games finite on one-side]
Consider a definable stochastic game and assume that one of the player has a finite set of strategies. Then the game has a uniform value.
\end{corollary}
\begin{proof}{Proof.}
It suffices to observe that the mixed extension of the game is both separable and definable, and to apply the previous theorem.\\ One could alternatively observe that the mixed extension fulfills the convexity assumptions of Proposition~\ref{convexparametric}. This shows that the Shapley operator of the game is definable, hence Corollary~\ref{c:shapdef} applies and yields the result. 
\end{proof}

\smallskip

The above theorems  generalize in particular the results of Bewley-Kohlberg \cite{BewlKohl76}, Mertens-Neyman \cite{MertNeym81} on {\em finite }stochastic games.\\

As shown by the following result, it is not true in general that semi-algebraic stochastic games have a semi-algebraic Shapley operator.
\begin{example}\label{exshap}
Consider the following stochastic game with two states $\{\omega_1,\omega_2\}$ and action sets $[0,1]$ for each player. The first state is absorbing with payoff $0$, while for the second state, the payoff is $$g(x,y,\omega_2)=\frac{1+x}{2(1+xy)^2}$$ and the transition probability is given by $$1-\rho(\omega_1|x,y,\omega_2)=\rho(\omega_2|x,y,\omega_2)=\frac{(1+x)y}{2(1+xy)^2},$$
for all $(x,y)$ in $[0,1]^2$.

 This stochastic game is defined by semi-algebraic and continuous functions but neither the Shapley operator $\Psi$ nor the curve of values  $(v_\lambda)_{\lambda\in(0,1]}$ are semi-algebraic mappings.
\end{example}
\begin{proof}{Proof.}
Notice first that $\rho(\omega_2|x,y,\omega_2)\in [0,1]$ for all $x$ and $y$ so the game is well defined.
It is straightforward that $\Psi_1(f_1,f_2)=f_1$, and $\Psi_2(f_1,f_2)=f_1+V(f_2-f_1)$ (where $V$ is the value of the parametric game in Lemma \ref{lemmecontreexemple})  hence $\Psi$ is not semi algebraic.

For any $\lambda\in]0,1[$ let $u_\lambda=\left(0,\frac{\lambda(e^\frac{1-\lambda}{2}-1)}{1-\lambda}\right)$, the identity $u_\lambda=v_\lambda$ will follow as we prove that $u_{\lambda}=\lambda\Psi(\frac{1-\lambda}{\lambda}\,u_{\lambda})$. This is clear for the first coordinate, and for the second, since $\frac{1-\lambda}{\lambda}\,u_{\lambda}=e^\frac{1-\lambda}{2}-1\in]0,1[$, Lemma \ref{lemmecontreexemple} implies that
\begin{eqnarray*}
\lambda \Psi_2(\frac{1-\lambda}{\lambda}\,u_{\lambda})&=&\lambda V(e^\frac{1-\lambda}{2}-1)\\
&=&\lambda \frac{e^\frac{1-\lambda}{2}-1}{1-\lambda}\\
&=&u_\lambda. \qquad \qquad  \end{eqnarray*}

\end{proof}

\begin{remark}{\rm As in Lemma~\ref{lemmecontreexemple}, one observes that both the Shapley operator $\Psi$ and the curve of values  $(v_\lambda)_{\lambda\in(0,1]}$ are globally subanalytic. }
\end{remark}

\subsection{Stochastic games with separable definable transitions}

This section establishes, by means of the Weierstrass density Theorem, that the assumptions we made on payoff functions can be brought down to mere continuity without altering our results on uniform values. From a conceptual viewpoint this shows that the essential role played by definability in our framework is to  tame oscillations generated by the underlying stochastic process $\rho$.

\begin{theorem}[Games with separable definable transitions]\label{t:value}
Let $(\Omega, X,Y,g,\rho)$ be a stochastic game, and assume that:
\begin{enumerate}
\item[(i)] $\Omega$ is finite and $X, Y$ are definable,
\item[(ii)] the reward function $g$ is an arbitrary continuous function,
\item[(iii)] the transition  function $\rho$ is definable and separable (e.g.\ polynomial).
\end{enumerate}
Then the game $(\Omega, X,Y,g,\rho)$ has a uniform value.
\end{theorem}

As it appears below, the proof of the above theorem relies on  Mertens-Neyman uniform value theorem \cite{MertNeym81} that we do not reproduce here. We shall however provide a complete proof of a weaker result in the spirit of the ``asymptotic approach" of Rosenberg-Sorin:
\begin{theorem}[Games with separable definable transitions -- weak version]\label{t:valuebis}
We consider a stochastic game $(\Omega, X,Y,g,\rho)$ which is as in Theorem~\ref{t:value}.\\
Then  the following limits exist and coincide:
$$\lim_{n\rightarrow=\infty}v_n=\lim_{\lambda \rightarrow 0}v_{\lambda}.$$
\end{theorem}

Before establishing the above results, we need some abstract results that allow to deal with certain approximation of  stochastic games. In the following proposition, the space $({\mathcal X},\|\cdot\|)$ denotes a real Banach space and $\K$ denotes a nonempty closed cone of ${\mathcal X}$. Being given two mappings $\Phi_1,\Phi_2:\K\rightarrow\K$, we define their
 supremum ``norm" through
 $$\| \Phi_1-\Phi_2 \|_{\infty}=\sup\left\{\|\Phi_1(f)-\Phi_2(f)\|: f\in \K\right\}.$$
 Observe that the above value may be $+\infty$, so that $\|\cdot\|_\infty$ is not a  norm, however, $\delta(\Phi_1,\Phi_2):= \|\Phi_1 -\Phi_2\|_{\infty}
/ (1+ \|\Phi_1 -\Phi_2\|_{\infty})$ does provide a proper metric (\footnote{We of course set:   $\delta(\Phi_1,\Phi_2):=1$ whenever $\| \Phi_1-\Phi_2 \|_{\infty}=\infty$.}) on the space of mappings $\K \to \K$.
We say that a sequence $\Psi_k:\K \rightarrow \K$ ($k\in\N$)  converges uniformly to $\Psi:\K\rightarrow\K$ if $\| \Psi_k-\Psi \|_{\infty}$ tends to zero as $k$ goes to infinity, or equivalently, if it converges
to $\Psi$ with respect to the metric $\delta$.
The observation that the set of nonexpansive mappings $\Psi:\K\to \K$
such that the limit $\lim_{n\to \infty}\Psi^n(0)/n$ does exist is closed in the
topology of uniform convergence was made in~\cite{gg04}.

 \begin{proposition}\label{t:unif} Let  $\Psi_k:\K \rightarrow \K$ be a sequence of nonexpansive mappings. Assume that\\
 (i) There exists $\Psi:\K \rightarrow \K$ such that $\Psi_k$ converges uniformly to $\Psi$ as $k\rightarrow+\infty$,\\
  (ii) for each fixed integer $k$, the sequence $\frac{1}{n}\Psi^n_k(0)$ has a limit $v^k$ in $\K$ as $n\rightarrow+\infty$.

 Then the sequence $v^k$ has a limit $v$ in $\K$, $\Psi$ is nonexpansive and $\frac{1}{n}\Psi^n(0)$ converges to $v$ as $k$ goes to infinity.
 \end{proposition}
\begin{proof}{Proof.}  Take $\epsilon>0$. Note first, that if $\Phi_1,\Phi_2$ are two nonexpansive mappings such that $\| \Phi_1-\Phi_2 \|_{\infty}\leq \epsilon$, we have $\| \Phi_1^n-\Phi_2^n \|_{\infty}\leq n\epsilon$. This follows indeed from an induction argument. The result obviously holds for $n=1$, so assume that $n\geq 2$ and consider that the inequality holds at $n-1$. For all $f$ in $\K$,  we have
\begin{eqnarray} \nonumber
\|\Phi_1^n(f)-\Phi_2^n(f)\|& \leq &\|\Phi_1(\Phi_1^{n-1}(f))-\Phi_1(\Phi_2^{n-1}(f))\| +\|\Phi_1(\Phi_2^{n-1}(f))-\Phi_2(\Phi_2^{n-1}(f))\|\\ \nonumber
& \leq &\|\Phi_1^{n-1}(f)-\Phi_2^{n-1}(f)\|+\epsilon\\
 & \leq & n\epsilon.\label{ineq}
\end{eqnarray}
Let us now prove that $v^k$ is a Cauchy sequence.  Let $N>0$ be such that $\| \Psi_p-\Psi_q \|_{\infty}\leq \epsilon$, for all $p,q\geq N$. Then, for each $p,q\geq N$ and  each positive integer $n$, we have
\begin{equation*}
\|\frac{\Psi_p^n(0)}{n}-\frac{\Psi_q^n(0)}{n}\| \leq  \epsilon.
\end{equation*}
Letting $n$ goes to infinity ($p$ and $q$ are fixed), one gets $\|v^p-v^q\|  \leq\epsilon$ and thus $v^k$ converges to  a vector $v$ belonging to $\K$.

Take $\epsilon>0$. Let $N$ be such that $\| \Psi_p-\Psi \|_{\infty}\leq \epsilon/3$ and $\|v^p-v\|<\epsilon/3$ for all $p\geq N$. Using \eqref{ineq}, one obtains  $\|\Psi_p^n(0)-\Psi^n(0)\|\leq n\,\epsilon/3$ where $n>0$ is an arbitrary integer. Whence
\begin{eqnarray*}
\|v-\frac{\Psi^n(0)}{n}\| & \leq & \|v-v^p\|+\|v^p-\frac{\Psi_p^n(0)}{n}\|+\|\frac{\Psi_p^n(0)}{n}-\frac{\Psi^n(0)}{n}\|\\
 & \leq & \frac{2\epsilon}{3}+\|v^p-\frac{\Psi_p^n(0)}{n}\|,\\
\end{eqnarray*}
for all $n>0$.
The conclusion follows by choosing $n$ large enough. 
\end{proof}

\smallskip

\noindent
Similarly, we prove:
\begin{proposition}\label{t:uniflambda} Let  $\Psi_k:\K \rightarrow \K$ be a sequence of nonexpansive mappings. Assume that\\
 (i) There exists $\Psi:\K \rightarrow \K$ such that $\Psi_k$ converges uniformly to $\Psi$ as $k\rightarrow+\infty$,\\
  (ii) for each fixed integer $k$, the family of fixed point $v_\lambda^k:=\lambda \Psi_k\left(\frac{1-\lambda}{\lambda}v_\lambda^k\right)$ has a limit $v^k$ in $\K$ as $\lambda\rightarrow 0$.

 Then the sequence $v^k$ has a limit $v$ in $\K$, $\Psi$ is nonexpansive and $v_\lambda:=\lambda \Psi\left(\frac{1-\lambda}{\lambda}v_\lambda\right)$ converges to $v$ as $k$ goes to infinity.
 \end{proposition}
\begin{proof}{Proof.}
Take $\epsilon>0$. Let $N>0$ be such that $\| \Psi_p-\Psi_q \|_{\infty}\leq \epsilon$, for all $p,q\geq N$. Then, for each $p,q\geq N$ and  any $\lambda\in]0,1]$, we have
\begin{eqnarray*}
\|v_\lambda^p-v_\lambda^q\|& = &\lambda \left\|\Psi_p\left(\frac{1-\lambda}{\lambda}v_\lambda^p\right)-\Psi_q\left(\frac{1-\lambda}{\lambda}v_\lambda^q\right)\right\|\\
& \leq &\lambda  \left\|\Psi_p\left(\frac{1-\lambda}{\lambda}v_\lambda^p\right)-\Psi_q\left(\frac{1-\lambda}{\lambda}v_\lambda^p\right)\right\| +\lambda \left\|\Psi_q\left(\frac{1-\lambda}{\lambda}v_\lambda^p\right)-\Psi_q\left(\frac{1-\lambda}{\lambda}v_\lambda^q\right)\right\|\\
    & \leq &\lambda \epsilon +(1-\lambda)\|v_\lambda^p-v_\lambda^q\|.
\end{eqnarray*}
so $\|v_\lambda^p-v_\lambda^q\|\leq\epsilon$.

Letting $\lambda$ to 0, we get that $v^k$ is a Cauchy sequence, hence converges to some $v$. Moreover, for any $p>N$,
\begin{eqnarray*}
\|v-v_\lambda\| & \leq & \|v-v^p\|+\|v^p-v_\lambda^p\|+\|v_\lambda^p-v_\lambda\|\\
 & \leq & 2\epsilon + \|v^p-v_\lambda^p\|
\end{eqnarray*}
for all $\lambda\in]0,1]$. Hence $v_\lambda$ converges to $v$. 
\end{proof}

\begin{proof}{[Proof of Theorem~\ref{t:valuebis}]} Let $k$ be a positive integer. {From} the Stone-Weierstrass theorem (see \cite{Choq66}), there exists a finite family $\{\pi_k(\cdot,\omega);\omega\in\Omega\}$ of real polynomial functions
\begin{equation}\label{eqpi}
\pi_k(x,y,\omega)=\sum_{i,\; j \mbox{ multi-index lower than } \, \delta^{\omega}_k} m^k_{ij}(\omega) \, x^i y^j
\end{equation}
with $\delta_k^{\omega}$ in $\N^*$, $m_{ij}^k(\omega)$ in $\R$ and $(x,y)$ in $X\times Y\subset\R^p\times \R^q$, such that
$$\sup_{\omega\in\Omega}\sup\left\{|\pi_k(x,y,\omega)-r(x,y,\omega)|:(x,y)\in X\times Y\right\}\;\leq\; \frac{1}{k}.$$
Consider now, for each positive $k$, the game given by $(\Omega,X,Y,\pi_k,\rho)$. Since this game is definable,  Proposition~\ref{p:sep} applies and the game has a value. In other words its Shapley operator $\Psi_k:\R^d\rightarrow\R^d$ (recall that the cardinality of $\Omega$ is $d$) is such that the sequence  $\frac{1}{n}\Psi_k^n(0)$ has a limit as $n$ goes to $+\infty$.
On the other hand, one easily sees that
$$\Psi(f)-\frac{1}{k}\leq \Psi_k(f) \leq \Psi(f) +\frac{1}{k}$$
whenever $f$ is in $\R^d$ and $k$ is positive. This proves that $\Psi_k$ converges uniformly to $\Psi$. Thus by using Proposition~\ref{t:unif} and Proposition~\ref{t:uniflambda} , we obtain the existence of a common limit $v$ in $\R^d$ of the sequence  $v_n=\frac{1}{n}\Psi^n(0)$ and of the family of fixed points $v_\lambda$.
\end{proof}


Let us now establish the stronger version of our result.

\begin{proof}{[Proof of Theorem \ref{t:value}]}
Let $k$ be a positive integer. As before we consider a finite family
of real polynomial functions, $\{\pi_k(\cdot,\omega);\omega\in\Omega\}$,
such that
\begin{equation}\label{eqpi2}
\sup_{\omega\in\Omega}\sup\left\{|\pi_k(x,y,\omega)-r(x,y,\omega)|:(x,y)\in X\times Y\right\}\;\leq\; \frac{1}{k}.
\end{equation}
Consider now, for each positive $k$, the game $\Gamma^k$ given by $(\Omega,X,Y,\pi_k,\rho)$. Since this game is definable,  Theorem~\ref{t:sepgamesunif} applies and the game has a uniform value $v^k$. Hence, there exists an integer $N$ (depending on $k$) and a strategy $\sigma$ of Player 1 which is $\frac{1}{k}$ optimal in the $n$-stage game $\Gamma^k_n$ for any $n\geq N$. That is, for any strategy $\tau$ of Player 2 and any starting state $\omega$,
\[
\gamma^k_n(\sigma,\tau,\omega)\geq v^k(\omega)-\frac{1}{k}.
\]
Hence by \eqref{eqpi2},
\begin{equation}\label{eqgammavk}
\gamma_n(\sigma,\tau,\omega)\geq v^k(\omega)-\frac{2}{k}.
\end{equation}
Taking the infimum over all possible strategies $\tau$, we get that for every $\omega$ and every large $n$,
 \[
 v_n(\omega)\geq v^k(\omega)-\frac{2}{k}.
 \]
Using the dual inequality
\begin{equation}
v_n(\omega)\leq v^k(\omega)+\frac{2}{k}\label{eqvnvk}
\end{equation}
one gets that $\limsup v_n(\omega)-\liminf v_n(\omega)\leq \frac{4}{k}$. Hence $v_n$ converges to some $v$. Moreover, combining  \eqref{eqgammavk} and \eqref{eqvnvk} yields
\[
\gamma_n(\sigma,\tau,\omega)\geq v_n(\omega)-\frac{4}{k}\geq v(\omega)-\frac{5}{k}
\]
for $n$ sufficiently large. Hence $v$ is the uniform value of the game. 
\end{proof}

\medskip

An immediate consequence of Theorem~\ref{t:value} is the following (\footnote{After this article was first submitted, examples were constructed in \cite{Ziliotto13, sorinvigeral13reversibility} that show that the definability assumption for the games described in this corollary cannot be removed.})
\begin{corollary}\label{coroswitch}
Any game with a definable transition probability, and either switching control or finitely many actions on one side, has a uniform value.
\end{corollary}

\subsection{Geometric growth in nonlinear Perron-Frobenius theory}\label{sec-geom}
We finally point out an application of the
present results to nonlinear Perron-Frobenius theory,
in which Shapley operators do appear, albeit after
a change of variables, using ``log-glasses~\cite{viro}.
In this setting, the mean payoff of the game determines the growth rate
of a population model.
The same Shapley operators arise in risk-sensitive control, where
the mean payoff problem is also of interest. Whereas
the importance of the o-minimal model of real semi-algebraic sets
is well known in game theory~\cite{bewkohl,Ney03}, the present application
show that there are natural Shapley operators which are definable
in a larger structure, the log-exp o-minimal model.

We denote by $C=\mathbb{R}_+^d$ the standard (closed)
nonnegative cone of $\mathbb{R}^d$, equipped with the
product ordering. We are interested
in maps $T$ defined on the interior of $C$,
satisfying some of the following properties.
We say that $T$ is {\em order preserving} if
\[
f\leq g\implies T(f)\leq T(g), \qquad \forall f,g\in \operatorname{int} C,
\]
that it is
{\em positively homogeneous} (of degree $1$) if
\[
T(\lambda f) = \lambda T(f), \qquad \forall f\in \operatorname{int}C, \, \forall \lambda >0,
\]
and {\em positively subhomogeneous} if
\[
T(\lambda f) \leq  \lambda T(f), \qquad \forall f\in \operatorname{int}C,\, \forall \lambda \geq 1.
\]
Let $\log: \operatorname{int} C\to \mathbb{R}^d$ denote the map which
does $\log$ entrywise, and let $\exp:= \log^{-1}$.
It is clear that $T$ is order-preserving and positively homogeneous
if and only if the conjugate map
\begin{align}\label{e-fromTtoShapley}
\Psi:= \log \circ\, T \circ \exp
\end{align}
is order-preserving and commutes with the addition of a constant.
These two properties
hold if and only if $\Psi$ is a dynamic programming
operator
associated to an undiscounted game with state space $\{1,\dots,d\}$,
i.e.\ if $\Psi$ can be written as in~\eqref{op_shapley}, but with
possibly noncompact sets of actions (see in particular~\cite{kolokoltsov}).
Note also that if $T$ is order preserving
and positively subhomogeneous, then, $\Psi$ is sup-norm
nonexpansive.

In the setting of nonlinear Perron-Frobenius theory, we are
interested in the existence of the geometric {\em growth rate} $\chi(T)$, defined by
\begin{align}
\chi(T):= \exp(\lim_{n\to \infty} n^{-1}\log T^n(e)) = \exp( \lim_{n\to \infty} n^{-1} \Psi^n(\log e))\label{e-corresp}
\end{align}
where $e$  is an arbitrary vector in the interior of $C$.

Problems of this nature arise in population dynamics. In this
context, one considers a population vector $f(n) \in \operatorname{int}\mathbb{R}_+^d$, where $[f(n)]_i$ represents the number of individuals
of type $i$ at  time $n$, assuming a dynamics of the form $f(n)=T(f(n-1))$.
Then, $[\chi(T)]_i= \lim_{n\to \infty} [T^n(f(0))]_i^{1/n}$ represents
the geometric growth rate of individuals of type $i$.


\begin{corollary}[Geometric Growth]\label{cor-geom}
Let $T$ be an order preserving and positively subhomogeneous
self map of $\operatorname{int} C$
that is definable in the log-exp structure, and let $e$ be a vector in $\operatorname{int}C$.
Then, the growth rate $\chi(T)$, defined by~\eqref{e-corresp}, does exist and is independent of the choice of~$e$.
\end{corollary}
\begin{proof}{Proof.} Apply Theorem~\ref{conv} to the operator~\eqref{e-fromTtoShapley},
which is nonexpansive in the sup-norm as well as definable in the log-exp structure, and use~\eqref{e-corresp}. 
\end{proof}

Here is now an application of Corollary~\ref{cor-geom} to
a specific class of maps.
\begin{corollary}[Growth minimization]
Assume that $T$ is a self-map of $\operatorname{int} C$
every coordinate of which can be written as
\begin{align}
[T(f)]_i = \inf_{p\in \mathcal{M}_i } \langle p,f\rangle  \qquad 1\leq i\leq d,
\label{e-T}
\end{align}
where $\mathcal{M}_i$ is a subset of $C$. Assume
in addition that each set $\mathcal{M}_i$
is definable in the log-exp structure.
Then, the growth rate $\chi(T)=\exp(\lim_{n\to \infty} n^{-1}\log T^n(e))$ does exist
and is independent of the choice of $e\in \operatorname{int} C$.
\end{corollary}
\begin{proof}{Proof.}
The map $T$ is obviously order preserving, positively homogeneous,
and, by Proposition~\ref{p:FO1} or Example~\ref{ex}, it is definable in the log-exp structure
as soon as every set $\mathcal{M}_i$ is definable in this structure.
Hence, the result
follows from Corollary~\ref{cor-geom}. 
\end{proof}

\smallskip

Several motivations lead to consider maps of the form~\eqref{e-T}.
The first motivation arises from discrete time controlled growth processes.
As above, to each time $n\geq 1$ and state $1\leq i\leq d$ is attached a population
$[f(n)]_i$. The control
at time $n$ is chosen after observing
the current state $1\leq i\leq d$. It consists
in selecting a vector $p\in \mathcal{M}_i$. Then, the population
at time $i$ becomes $[f(n)]_i =\langle p, f(n-1)\rangle$. The iterate $[T^n(e)]_i$
represents the minimal possible population at state $i$ and time $n$,
with an initial population $e$. Then, the limit $\chi(T)$
represents the minimal possible growth rate.
This is motivated in particular by some therapeutic problems
(see e.g~\cite{billy}), for which $\chi(T)$ yields
a lower bound on the achievable growth rates.

Another motivation comes from risk sensitive control~\cite{fleming,hernandez} or from  mathematical finance models
with logarithmic utility~\cite{akian}. In this context,
it is useful to consider the conjugate map $\Psi:= \log \circ \,T \circ \exp$,
which has
the following explicit representation
\begin{align}
[\Psi(h)]_i & = \inf_{p\in \mathcal{M}_i } \log(\sum_{1\leq j\leq d} p_j e^{h_j})
%
= \inf_{p\in \mathcal{M}_i } \sup_{q\in \Delta_d}
(-S(q,p) + \langle q,h\rangle )\label{e-explicit}
\end{align}
where
\[S(q,p):= \sum_{1\leq j \leq d} q_j \log(q_j/p_j)
\]
denotes the {\em relative entropy} or {\em Kullback-Leibler divergence},
and $\Delta_d:=\{q\in C\mid \sum_{1\leq j\leq d} q_j=1\}$
is the standard simplex.
Then, $\log [\chi(T)]_i$
can be interpreted as the value of an ergodic risk sensitive problem,
and it is also the value of a zero-sum game.

The case in which inf is replaced by sup in~\eqref{e-T}, i.e.,
$[T(f)]_i = \sup_{p\in \mathcal{M}_i } \langle p,f\rangle$,
for $1\leq i\leq d$, which is also of interest, turns out to be simpler.
Indeed, each coordinate of the operator $\Psi:= \log \circ \,T \circ \exp$ becomes convex (this can be easily seen
from the representation analogous to~\eqref{e-explicit},
in which the infimum is now replaced by a supremum).
More generally, the latter convexity property is known to hold
if and only if $\Psi$ is the dynamic programming operator
of a one player stochastic game~\cite{spectral,vigeral}.
It has been shown by several authors~\cite{gg04,vigeral,Rena11} that
for this class of operators (or games), the limit
$\lim_{n\rightarrow+\infty} \Psi^n(f)/n$ does exist,
from which the existence of the limit~\eqref{e-corresp} readily
follows.

Finally, we note that we may consider more general hybrid versions of~\eqref{e-T},
for instance with a partition $\{1,\dots,d\}=I\cup J$ and
\[
[T(f)]_i  = \inf_{p\in \mathcal{M}_i } \langle p,f\rangle  \qquad i\in I,
\qquad [T(f)]_i  = \sup_{p\in \mathcal{M}_i } \langle p,f\rangle  \qquad i\in J \enspace .
\]
Then the existence of the growth rate, for such maps, also
follows from Corollary~\ref{cor-geom}.

\bigskip

\noindent


%


%
%
%




\section*{Acknowledgments.}
The authors would like to thank J. Renault, S. Sorin and X. Venel for their very useful comments.

\bibliography{ominimal}
\bibliographystyle{amsplain.bst}





\end{document}